\documentclass[a4paper]{article}
\usepackage{}
\usepackage{bbm}
  \usepackage{paralist}
\usepackage{latexsym}
\usepackage{amsmath}
\usepackage{amssymb}
\usepackage{bbding}
\usepackage{cite}
\usepackage{yhmath}
\usepackage{amsthm}
\usepackage{amsmath,amssymb,mathrsfs}
\usepackage{amsfonts}
\usepackage{graphics}
\usepackage{epsfig}
\usepackage{indentfirst}
\usepackage{dcolumn}
\usepackage{indentfirst}
\usepackage{color}
\usepackage{bbding}
\usepackage{cite}
\usepackage{amsthm}
\usepackage{amssymb}
\usepackage{mathrsfs}
\usepackage{amsmath}
\usepackage{hyperref}
\usepackage{CJK}
\numberwithin{equation}{section}

\newtheorem{definition}{Definition}[section]

\title{Incompressible Limit of a Compressible Liquid Crystals System
\footnotetext[0]{2010 Mathematics Subject Classification. Primary: 76N10, 35Q35, 35Q30. }
\footnotetext[0]{This work was supported partly by NSFC grant 11071043, 11131005, 11071069.}
}
\author{Yi-hang Hao\footnote{{\it E-mail address}: 10110180022@fudan.edu.cn.}\qquad Xian-gao Liu\footnote{{\it Corresponding author, E-mail address}: xgliu@fudan.edu.cn,}
\\{ \it\small  School of Mathematic Sciences, Fudan University,
Shanghai, 200433, P.R.China }}
\date{}
\begin{document}
\maketitle
\begin{abstract}
This article is devoted to the study of the so-called incompressible limit for solutions of the compressible liquid crystals system. We consider the problem in the whole space $\mathbb{R}^{\mathbb{N}}$ and a bounded domain of $\mathbb{R}^{\mathbb{N}}$ with Dirichlet boundary conditions. Here the number of dimension $\mathbb{N}=2$ or $3$.
\end{abstract}
\noindent{\bf Key words:} compressible liquid crystals, incompressible liquid crystals, Mach number.
\newtheorem{prn}{\textbf{Proposition}}
\newtheorem{rk}{\textbf{Remark}}
\newtheorem{theorem}{\textbf{Theorem}}
\newtheorem{lemma}{\textbf{Lemma}}
\section {Introduction}
As been well known, Mach number is associated with the compressibility of fluids. So one can derive formally incompressible fluid equations from the compressible ones. This process is called the incompressible limit to solutions of the compressible fluids. For Navier-Stokes equations, there are a lot of results to the problem, see\cite{PlN,bepn,DG}. Here our work shows that the same results also hold to liquid crystals system. In this article, we want to pose a view of the incompressible limit, and give its proofs. We plan to prove this question in
the whole space and in a bounded domain with Dirichlet boundary conditions. Both of them are considered in $\mathbb{R}^{\mathbb{N}},\ \mathbb{N}=2,3$.
\par We recall that A.Majda \cite{maj} introduced the problem in the case of Euler equations. He did it using the non-dimensionlized method to the compressible fluid equations, and talked about the asymptotic properties of solutions to Euler equations. Considering our model is the complex fluid system, the
nondimensionlized method needs us to consider Mach number's inner physical
meaning. We plan to avoid this. In physics, Mach number is defined by $M=\frac{u_m}{\sqrt{\frac{dp}{d\rho}(\rho_m)}}$, where $p(\rho)=a\rho^\gamma$ is the pressure, $\rho_m$ and $u_m$ are $L^\infty$ norm of initial density and velocity respectively. In the macroscopic view, as $M\rightarrow0$, we can viewed the process in a large time scales, the fluid should behave like an incompressible one, the density is almost a constant, and the velocity is small, see \cite{PlN,bepn,DG}.
\par In the article \cite{LH,liu} , they have studied the existence of global weak solutions to the following compressible liquid crystals system.
\begin{align}
&\widetilde{\rho}_t+div(\widetilde{\rho}\widetilde{u})=0\label{comp1},\\
&(\widetilde{\rho} \widetilde{u})_t+div(\widetilde{\rho}\widetilde{u}\otimes\widetilde{u})-\widetilde{\mu}\Delta \widetilde{u}-\widetilde{\xi}\nabla div\widetilde{u}+\nabla \widetilde{p}\nonumber\\
&\qquad=\widetilde{\alpha}\nabla[(F(\widetilde{d})+\frac{1}{2}|\nabla\widetilde{d}|^2)I-\nabla\widetilde{d}\odot\nabla \widetilde{d}+\frac{1}{2\widetilde{\lambda}}(\widetilde{d}\otimes\widetilde{N}-\widetilde{N}\otimes\widetilde{d})],\label{comp2}\\
&\widetilde{d}_t+(\widetilde{u}\cdot\nabla)\widetilde{d}-\widetilde{\Omega}\widetilde{d}=\widetilde{\lambda}(\Delta \widetilde{d}-f(\widetilde{d})).\label{comp3}
\end{align}
where $\widetilde{\rho}\geq0$, $\widetilde{u}\ \in\ \mathbb{R}^\mathbb{N}$ and $\widetilde{d}\ \in\ \mathbb{R}^\mathbb{N}$ are the fluid density, velocity and molecular direction respectively, $I$ denotes a unite matrix, $\widetilde{p}=a\widetilde{\rho}^\gamma$ the pressure, and $\widetilde{N}=\widetilde{d}_t+(\widetilde{u}\cdot\nabla)\widetilde{d}-\widetilde{\Omega}\widetilde{d}$ the director movement in satellited coordinates. $\nabla\widetilde{d}\odot\nabla \widetilde{d}$ is a matrix whose $(i,j)$-th entry is given by $\nabla_i\widetilde{d}\cdot\nabla_j\widetilde{d}$. Also we denote $(\widetilde{d}\otimes \widetilde{N})_{ij}=\widetilde{d}_i\widetilde{N}_j$, $F(\widetilde{d})=\frac{1}{2{\zeta}^2}(|\widetilde{d}|^2-1)^2, f(\widetilde{d})=F'(\widetilde{d})$. Here we don't mention the properties of solutions to the above equations, the reader can refer to \cite{LH,liu}. Summing up what we mentioned above, we scale the variable of equations (\ref{comp1}) - (\ref{comp3}) in the following way:
\begin{align}
\widetilde{x}=x,\quad \widetilde{t}=\epsilon t,\quad\widetilde{\rho}=\rho(x,\epsilon t),\quad \widetilde{u}=\epsilon u(x,\epsilon t),\quad \widetilde{d}=d(x,\epsilon t),\nonumber
\end{align}
and the coefficients are scaled as
\begin{align}
\widetilde{\mu}=\epsilon\mu_\epsilon,\quad \widetilde{\xi}=\epsilon\xi_\epsilon,\quad \widetilde{\lambda}=\epsilon\lambda_\epsilon,\quad \widetilde{\alpha}={\epsilon}^2\alpha_\epsilon.\nonumber
\end{align}
Following above assumption, system (\ref{comp1}) - (\ref{comp3}) become
\begin{align}
&\rho_t+div(\rho u)=0\nonumber,\\
&(\rho u)_t+div(\rho u\otimes u)-\mu_{\epsilon}\Delta u-\xi_{\epsilon}\nabla divu+\frac{1}{{\epsilon}^2}\nabla p\nonumber\\
&\qquad=\alpha_{\epsilon}\nabla[(F(d)+\frac{1}{2}|\nabla d|^2)I-\nabla d\odot\nabla d+\frac{1}{2\lambda_{\epsilon}}(d\otimes N-N\otimes d)],\nonumber\\
&d_t+(u\cdot\nabla)d-\Omega d=\lambda_{\epsilon}(\Delta d-f(d)).\nonumber
\end{align}
We also need the following assumptions,
\begin{align}
\mu_\epsilon\rightarrow \mu>0,\ \mu_\epsilon+\xi_\epsilon\rightarrow\mu+\xi >0,\ \alpha_\epsilon\rightarrow\alpha>0,\ \lambda_\epsilon\rightarrow\lambda>0.\nonumber
\end{align}
Without loss of generality, we assume that $\mu_\epsilon,\ \xi_\epsilon,\ \alpha_\epsilon,\ \lambda_\epsilon$ are independent of $\epsilon$. Letting
$\mu_\epsilon,\ \xi_\epsilon,\ \alpha_\epsilon,\ \lambda_\epsilon$ replace by $\mu,\ \xi,\ \alpha,\ \lambda$,
we consider the following equations:
\begin{align}
&\rho_t+div(\rho u)=0\label{com1},\\
&(\rho u)_t+div(\rho u\otimes u)-\mu\Delta u-\xi\nabla divu+\frac{1}{{\epsilon}^2}\nabla p\nonumber\\
&\qquad=\alpha\nabla[(F(d)+\frac{1}{2}|\nabla d|^2)I-\nabla d\odot\nabla d+\frac{1}{2\lambda}(d\otimes N-N\otimes d)],\label{com2}\\
&d_t+(u\cdot\nabla)d-\Omega d=\lambda(\Delta d-f(d)).\label{com3}
\end{align}
{\bf In the whole space case:}\\We set the initial conditions as follows:
\begin{align}
&\rho_\epsilon|_{t=0}=\rho_\epsilon^0,\quad \rho_\epsilon u_\epsilon|_{t=0}=m_\epsilon^0,\quad d_\epsilon|_{t=0}=d_\epsilon^0,\label{ini1}\\
&\rho_\epsilon^0\geq0,\quad m_\epsilon^0=0\ a.e\ {\rho_\epsilon^0=0},\quad|d_\epsilon^0|=1,\label{ini2}
\end{align}
and
\begin{align}
\int_{R^{\mathbb{N}}}\frac{1}{2}\rho_\epsilon^0|u_\epsilon^0|^2+\frac{\alpha}{2}|\nabla d_\epsilon^0|^2+\alpha F(d_\epsilon^0)+\frac{1}{\epsilon^2(\gamma-1)}[(\rho_\epsilon^0)^\gamma-\gamma(\rho_\epsilon^0-1)-1\leq C,\label{ini3}
\end{align}
where $C$ is independent of $\epsilon$. We assume that
\begin{align}
&\sqrt{\rho_\epsilon^0}u_\epsilon^0\rightarrow u_0\ in\ L^2(R^{\mathbb{N}}),\label{inilimit1}\\
&d_\epsilon^0\rightarrow d_0 \ in \ L^\infty((0,T)\times R^\mathbb{N})\cap\dot{H}^1(R^{\mathbb{N}})\cap\dot{H}^2(R^{\mathbb{N}}).\label{inilimit2}
\end{align}
In addition, we also assume that the reference density is equal to $1$, which means the expect density or the integral mean.
As $\epsilon\rightarrow 0$, we expect that $\rho_\epsilon\rightarrow1$ in some space and the system becomes
\begin{align}
&divu=0\label{incom1},\\
&u_t+div(u\otimes u)-\mu\Delta u+\nabla \pi\nonumber\\
&\qquad\qquad=\alpha\nabla[-\nabla d\odot\nabla d+\frac{1}{2\lambda}(d\otimes N-N\otimes d)],\label{incom2}\\
&d_t+(u\cdot\nabla)d-\Omega d=\lambda(\Delta d-f(d)).\label{incom3}
\end{align}
where $\pi$ is the limit of the $\frac{\rho_\epsilon^\gamma-1}{\epsilon^2}+F(d_\epsilon)+\frac{1}{2}|\nabla d_\epsilon|^2$ in some sense. We define Banach space $\mathcal{H}$ as follows
\begin{eqnarray}
\mathcal{H}(R^{\mathbb{N}})=\{\left.d\ \right|\ d\in L^\infty(R^{\mathbb{N}}),\ d\in \dot{H}^1(R^{\mathbb{N}})\}.\nonumber
\end{eqnarray}
The following Orlicz spaces $L_2^p(R^{\mathbb{N}})$(see\cite{lions2}, pp.288) are needed.
\begin{eqnarray}
L_2^p(R^{\mathbb{N}})=\left\{f\ \in\ L_{local}^1(R^{\mathbb{N}}),\ \left.f\right|_{|f|\leq\frac{1}{2}}\ \in\ L^2(R^{\mathbb{N}}),\ \left.f\right|_{|f|\geq\frac{1}{2}}\ \in\ L^p(R^{\mathbb{N}})\right\}.\nonumber
\end{eqnarray}
\begin{theorem}\label{whole}
Let $\gamma>\frac{3}{2}$, $\rho_\epsilon,\ u_\epsilon,\ d_\epsilon$ be a sequence of global weak solutions to (\ref{com1})--(\ref{com3}) in the whole space with initial conditions (\ref{ini1})--(\ref{ini3}). And in addition, we assume (\ref{inilimit1}), (\ref{inilimit2}).  Then, as $\epsilon\rightarrow0$ we have,
\begin{align}
&\rho_\epsilon\rightarrow1\quad in\ L^\infty(0,T;L^\gamma(R^{\mathbb{N}})),\ if\ \gamma\geq2,\nonumber\\
&\rho_\epsilon\rightarrow1\quad in\ L^\infty(0,T;L_2^\gamma(R^{\mathbb{N}})),\ if\ \gamma<2,\nonumber\\
&u_\epsilon\rightarrow u\quad in \ L^2((0,T)\times R^{\mathbb{N}}),\quad u_\epsilon\rightharpoonup u\qquad in\ L^2(0,T;H^1(R^{\mathbb{N}})),\nonumber\\
&d_\epsilon\rightharpoonup^* d\qquad in\ L^\infty(0,T;\mathcal{H}(R^{\mathbb{N}})),\quad d_\epsilon\rightarrow d\qquad in\ L^2(0,T;W^{1,s}(B_r)),\nonumber\\
&d_\epsilon\rightarrow d\qquad in\ C^0([0,T];L^s(B_r)),\nonumber
\end{align}
where $B_r$ is a bounded ball with radius $r$ centering the origin and
\begin{align}
s\in [2,\infty)\ if\ \mathbb{N}=2,\  s<\frac{2\mathbb{N}}{\mathbb{N}-2}\ if\ \mathbb{N}=3.\nonumber
\end{align}
Here $u,\ d,\ and\ some\ \pi$ satisfy (\ref{incom1})--(\ref{incom3}) in distributions of the whole space with initial conditions $u|_{t=0}=Pu_0,\ d|_{t=0}=d_0$.
\end{theorem}
\noindent{\bf In the case of bounded domain:}\\
Let $D$ denote a bounded domain of $\mathbb{R}^{\mathbb{N}}$ with ($C^{2+a},\ a>0$) boundary conditions. Beside
 the conditions in the whole space case, we give the boundary conditions:
\begin{align}
u_\epsilon=0,\quad d_\epsilon=d_\epsilon^0,\qquad on\ D\times(0,T).\label{ini8}
\end{align}
Also we need the following definition.
\begin{definition}
Let  $D \in R^\mathbb{N}$ be a bounded, connected open set with $C^2$ boundary. We call $D$ is satisfying $H$ if the following overdetermined problem only has trivial solution,
\begin{align}
\left\{\begin{array}{ll}
&-\Delta\phi=\lambda\phi,\ in\ D,\ \lambda\geq0,\\
&{\frac{\partial\phi}{\partial n}}|_{\partial D}=0,\ \phi|_{\partial D}=const.
\end{array}\right.\nonumber
\end{align}
\end{definition}
$H$-condition is defined by B. Desjardins, E. Grenier, P.-L. Lions. and N. Masmoudi in \cite{bepn}. In $\mathbb{R}^2$ $H$-condition is generic. But in $\mathbb{R}^3$, by Schiffer's conjecture, we
know that every bounded domain $D$ satisfies $H$ except the ball\cite{sch}.
\begin{theorem}\label{bounded}
In the Dirichlet boundary case, we have
\begin{align}
&\rho_\epsilon\rightarrow1\quad in\ L^\infty(0,T;L^\gamma(D)),\quad u_\epsilon\rightharpoonup u\qquad in\ L^2(0,T;H_0^1(D)),\nonumber\\
&d_\epsilon\rightharpoonup d\qquad in\ L^2(0,T;H^2(D)),\quad d_\epsilon\rightarrow d\qquad in\ L^2(0,T;W^{1,s}(D)),\nonumber\\
&d_\epsilon\rightarrow d\qquad in\ C^0([0,T];L^s(D)),\nonumber
\end{align}
and $u_\epsilon\rightarrow u\ in\ L^2((0,T)\times D)$ if $D$ satisfies $H$-condition.
$u,\ d,\ and\ some\ \pi$ satisfy (\ref{incom1}) - (\ref{incom3}) in $D'(D)$ with initial conditions $u|_{t=0}=Pu_0,\ d|_{t=0}=d_0$.
\end{theorem}
The remaining part of this article is organized as follows. We give some estimates satisfying the two cases in section \ref{2}; in section \ref{3}, we give the convergence of part of $u_\epsilon$ for the two cases; section \ref{four} and section \ref{5} will contribute to prove Theorem \ref{whole} and Theorem \ref{bounded} respectively.

\section{Some estimates from the energy inequality}\label{2}
\noindent{\bf The whole space case:}\\
Let us denote $E_\epsilon(t),\ E_\epsilon^0$ as
\begin{align}
E_\epsilon(t)=\int_{R^{\mathbb{N}}}\frac{1}{2}\rho_\epsilon|u_\epsilon|^2+\frac{\alpha}{2}|\nabla d_\epsilon|^2+\alpha F(d_\epsilon)+\frac{1}{\epsilon^2(\gamma-1)}[(\rho_\epsilon)^\gamma-\gamma(\rho_\epsilon-1)-1],\nonumber\\
E_\epsilon^0=\int_{R^{\mathbb{N}}}\frac{1}{2}\rho_\epsilon^0|u_\epsilon^0|^2+\frac{\alpha}{2}|\nabla d_\epsilon^0|^2+\alpha F(d_\epsilon^0)+\frac{1}{\epsilon^2(\gamma-1)}[(\rho_\epsilon^0)^\gamma-\gamma(\rho_\epsilon^0-1)-1].\nonumber
\end{align}
We can write the energy equation of system (\ref{com1})--(\ref{com3}) as follows
\begin{align}
E_\epsilon(t)+\int_0^T\int_{R^{\mathbb{N}}}\mu|\nabla u_\epsilon|^2+\xi|divu_\epsilon|^2+\frac{\alpha}{\lambda}|N_\epsilon|^2\leq
E_\epsilon^0\leq C.\label{energy}
\end{align}
From (\ref{energy}) we obtain
\begin{align}
&\|\sqrt{\rho_\epsilon}u_\epsilon\|_{L^2((0,T)\times R^{\mathbb{N}})}\leq C,\quad \|\nabla u_\epsilon\|_{L^2((0,T)\times R^{\mathbb{N}})}\leq C,\label{2.2}\\
&\sup_t\int_{R^{\mathbb{N}}}\frac{1}{\epsilon^2(\gamma-1)}[(\rho_\epsilon)^\gamma-\gamma(\rho_\epsilon-1)-1]\leq C,\label{2.3}\\
&\|\nabla d_\epsilon\|_{L^\infty((0,T);L^2(R^{\mathbb{N}}))}\leq C,\quad \|N_\epsilon\|_{L^2((0,T)\times R^{\mathbb{N}})}\leq C.\label{2.4}
\end{align}
Also we have
\begin{align}
x^\gamma-\gamma(x-1)-1\geq\nu_\delta|x-1|^\gamma,\ |x-1|\geq\delta,\ x\geq0,\ \mathrm{some}\ \nu_\delta>0.\label{2.5}
\end{align}
So by virtue of (\ref{2.3}) and (\ref{2.5}), we obtain
\begin{align}
\sup_t\int_{B_r}|\rho_\epsilon-1|^\gamma&\leq\sup_t\int_{B_r}|\rho_\epsilon-1|_{|\rho_\epsilon-1|\geq\delta}^\gamma+|\rho_\epsilon-1|_{|\rho_\epsilon-1|\leq\delta}^\gamma\nonumber\\
&\leq\sup_t\int_{B_r}\frac{\epsilon^2}{\nu_\delta}\frac{(\rho_\epsilon)^\gamma-\gamma(\rho_\epsilon-1)-1}{\epsilon^2}dx+|B_r|\delta\nonumber\\
&\leq\frac{\epsilon^2}{\nu_\delta}C+|B_r|\delta.\label{2.1}
\end{align}
Let $\epsilon\rightarrow0$, and then $\delta\rightarrow0$, yield $\rho_\epsilon\rightarrow1\ in\ L^\infty(0,T;L^\gamma({B_r}))$. We denote $\varphi_\epsilon=\frac{\rho_\epsilon-\overline{\rho_\epsilon}}{\epsilon}$. For all $x>0$, we have
\begin{align}
\left\{\begin{array}{ll}
&x^\gamma-1-\gamma(x-1)\geq\nu|x-1|^\gamma,\quad \gamma\geq2,\\
&x^\gamma-1-\gamma(x-1)\geq\nu|x-1|^2,\quad \gamma<2,\ |x-1|\leq \frac{1}{2},\\
&x^\gamma-1-\gamma(x-1)\geq\nu|x-1|^\gamma,\quad \gamma<2,\ |x-1|\geq \frac{1}{2}.
\end{array}\right.\label{2.6}
\end{align}
Using (\ref{2.3}) and (\ref{2.6}), we have
\begin{align}
\left\{\begin{array}{ll}
&\|\varphi_\epsilon\|_{L^\infty(0,T;L^\gamma(R^{\mathbb{N}}))}\leq C,\quad \gamma\geq 2,\\
&\|\varphi_\epsilon|_{|\rho_\epsilon-1|\leq\frac{1}{2}}\|_{L^\infty(0,T;L^2(R^{\mathbb{N}}))}\leq C,\quad \gamma< 2,\\
&\|\varphi_\epsilon|_{|\rho_\epsilon-1|\geq\frac{1}{2}}\|_{L^\infty(0,T;L^\gamma(R^{\mathbb{N}}))}\leq C\epsilon^{\frac{2}{\gamma}-1},\quad \gamma< 2.
\end{array}\right.\label{2.7}
\end{align}
Then, we obtain
\begin{align}
\varphi_\epsilon\in L^\infty(0,T;L^\gamma(R^{\mathbb{N}})),\quad \gamma\geq2,\label{2.8}\\
\varphi_\epsilon\in L^\infty(0,T;L^\gamma_2(R^{\mathbb{N}})),\quad \gamma<2,\label{2.9}
\end{align}
The above results also right to $\varphi_\epsilon|_{t=0}$. Next, we consider $u_\epsilon$.
Let split $u_\epsilon$ as follows:
\begin{align}
&u_\epsilon=u_\epsilon^1+u_\epsilon^2,\nonumber\\
&u_\epsilon^1=u_\epsilon|_{|\rho_\epsilon-1|\leq\frac{1}{2}},\quad u_\epsilon^2=u_\epsilon|_{|\rho_\epsilon-1|\geq\frac{1}{2}}.\nonumber
\end{align}
We have
\begin{align}
\sup_t\int_{R^{\mathbb{N}}}|u_\epsilon^1|^2dx\leq2\sup_t\int_{R^{\mathbb{N}}}\rho_\epsilon|u_\epsilon|^2dx\leq C\nonumber\\
\int_{R^{\mathbb{N}}}|u_\epsilon^2|^2dx\leq2\int_{R^{\mathbb{N}}}\epsilon|\varphi_\epsilon||u^2_\epsilon|^2dx\nonumber\\
\leq\epsilon\|\varphi_\epsilon\|_{L_2^\gamma(R^{\mathbb{N}})}\|u_\epsilon^2\|_{L^2(R^{\mathbb{N}})}^\theta\|\nabla u_\epsilon\|_{L^2(R^{\mathbb{N}})}^{1-\theta},\nonumber
\end{align}
where $\theta=0$ if $\mathbb{N}=2$ and $\theta=\frac{1}{2}-\frac{3}{4\gamma}$ if $\mathbb{N}=3$. We have deduced that
\begin{align}
\|u_\epsilon^1\|_{L^\infty(0,T;L^2(R^{\mathbb{N}}))}\leq C,\label{2.10}\\
\|u_\epsilon^2\|_{L^2((0,T)\times R^{\mathbb{N}})}\leq\epsilon^{\frac{1}{2-\theta}}C.\label{2.11}
\end{align}
So we have $\|u_\epsilon\|_{L^2(0,T;H_0^1(R^{\mathbb{N}}))}\leq C$. Let's turn to $d_\epsilon$.
Equation (\ref{com3}) multiplied with $d$ yields
\begin{align}
|d|_t^2-(u\cdot\nabla)|d^2|-\gamma_\epsilon\Delta|d|^2+\frac{\gamma_\epsilon}{{\zeta}^2}(|d|^2-1)|d|^2\geq0.\nonumber
\end{align}
That is
\begin{align}
(|d|^2-1)_t-(u\cdot\nabla)(|d|^2-1)-\gamma_\epsilon\Delta(|d|^2-1)+\frac{\gamma_\epsilon}{{\zeta}^2}(|d|^2-1)|d|^2\geq0.\nonumber
\end{align}
Noting that $u_\epsilon\in L^2(0,T;H_0^1(R^{\mathbb{N}}))$ and using maximum principle, we have
\begin{align}
|d_\epsilon|\leq|d_\epsilon^0|=1.\label{2.12}
\end{align}
Then, we obtain $d_\epsilon\in L^\infty(0,T;\mathcal{H}(R^{\mathbb{N}}))$.
Using $\Delta d_\epsilon-f(d_\epsilon)=\frac{1}{\lambda}N_\epsilon$, $\|f(d_\epsilon)\|_{L^2{((0,T)\times R^\mathbb{N})}}$ and (\ref{2.12}), we have
\begin{align}
 \|\nabla^2d_\epsilon\|_{L^2((0,T)\times R^\mathbb{N})}&\leq C(\|\Delta d_\epsilon\|_{L^2((0,T)\times R^\mathbb{N})}+\|d_\epsilon\|_{L^2((0,T)\times R^\mathbb{N})}+\|\nabla^2d_{\epsilon,0}\|_{L^2((0,T)\times R^\mathbb{N})})\nonumber\\
 &\leq C.\nonumber
\end{align}
Noticing $N_\epsilon=d_t+(u_\epsilon\cdot\nabla)d_\epsilon-\Omega_\epsilon d_\epsilon$ is bounded in $L^2((0,T)\times R^\mathbb{N})$, we have
\begin{align}
\|d_{\epsilon t}\|_{L^2(0,T;L^{\frac{3}{2}}(B_r))}\leq C(r).\nonumber
\end{align}
Thus, we have obtained
\begin{align}
&\rho_\epsilon\rightarrow1\qquad in\ L^\infty(0,T;L^\gamma(R^{\mathbb{N}})),\ if\ \gamma\geq2,\nonumber\\
&\rho_\epsilon\rightarrow1\qquad in\ L^\infty(0,T;L_2^\gamma(R^{\mathbb{N}})),\ if\ \gamma<2,\nonumber\\
&u_\epsilon\rightharpoonup u\qquad in\ L^2(0,T;H^1(R^{\mathbb{N}})),\nonumber\\
&d_\epsilon\rightharpoonup d\qquad in\ L^\infty(0,T;\mathcal{H}(R^{\mathbb{N}})),\nonumber\\
&d_\epsilon\rightarrow d\qquad in\ L^2(0,T;W^{1,t}(B_r)),\nonumber\\
&d_\epsilon\rightarrow d\qquad in\ C^0([0,T];L^t(B_r)),\nonumber
\end{align}
where $t\in\ [2,\infty)$ if $\mathbb{N}=2$ and $t<\frac{2\mathbb{N}}{\mathbb{N}-2}$ if $\mathbb{N}=3$.
So we can easily get the convergence of terms in equations (\ref{com1}) - (\ref{com3}) (see \cite{LH,f2,liu}), except $\rho_\epsilon u_\epsilon\otimes u_\epsilon$. Here, $\pi$ is the limit of the $\frac{\rho_\epsilon^\gamma-1}{\epsilon^2}+F(d_\epsilon)+\frac{1}{2}|\nabla d_\epsilon|^2$ in some sense.\\
{\bf The bounded domain case:}\\
The results we obtained in the whole space case also hold in the bounded domain with both $R^{\mathbb{N}}$ and $B_r$ replaced by $D$.
\par At the end of this section, we set
\begin{align}
Q=\nabla\Delta^{-1}div,\ and\ P=I-Q,\nonumber
\end{align}
where $\Delta^{-1}$ defined by fourier transform in the whole space $R^{\mathbb{N}}$ and by the followings in a bounded domain $D$ with Neummann boundary conditions:
\begin{align}
f=\Delta^{-1}g,\ if\ \Delta f=g,\ \left.{\frac{\partial f}{\partial n}}\right|_{\partial D}=0,\ \int_Df=0.\nonumber
\end{align}
So we can split $u_\epsilon=Pu_\epsilon+Qu_\epsilon$ and consider their convergence, respectively.
\section{The convergence of $Pu_\epsilon$}\label{3}
Using operator $P$ and $Q$, we can split equations (\ref{com2}) as
\begin{align}
&P(\rho_\epsilon u_\epsilon)_t+divP(\rho_\epsilon u_\epsilon\otimes u_\epsilon)-\mu\Delta Pu_\epsilon=\alpha P\nabla[-\nabla d_\epsilon\odot\nabla d_\epsilon+\frac{1}{2\lambda}(d_\epsilon\otimes N_\epsilon-N_\epsilon\otimes d_\epsilon)],\label{pcom1}\\
&(Q(\rho_\epsilon u_\epsilon))_t+divQ(\rho_\epsilon u_\epsilon\otimes u_\epsilon)-\xi\nabla divu_\epsilon-\mu\Delta Qu_\epsilon\nonumber\\
&\qquad+\frac{a}{{\epsilon}^2}\nabla((\rho_\epsilon)^\gamma-\gamma(\rho_\epsilon-1)-1)
+\frac{a\gamma}{{\epsilon}^2}\nabla(\rho_\epsilon-\overline{\rho_\epsilon})\nonumber\\
&\qquad=\alpha Q\nabla[(F(d_\epsilon)+\frac{1}{2}|\nabla d_\epsilon|^2)I-(\nabla d_\epsilon\odot\nabla d_\epsilon)+\frac{1}{2\lambda}(d_\epsilon\otimes N_\epsilon-N_\epsilon\otimes d_\epsilon)].\label{qcom2}
\end{align}
Equation (\ref{com1}) can be rewritten
\begin{align}
\epsilon(\varphi_\epsilon)_t+div(Q(\rho_\epsilon u_\epsilon))=0.\label{qcom1}
\end{align}
Here we need the following lemma, the proof can be find in \cite{lions2}({Lemma 5.1}).
\begin{lemma}\label{lemma2}
Let $g^k,\ h^k$ converge weakly to $g,\ h$ respectively in $L^{p_1}(0,T;L^{p_2}(D)),\ L^{q_1}(0,T;L^{q_2}(D))$, where $1\leq p_1,\ p_2\leq+\infty$,
\begin{align}
\frac{1}{p_1}+\frac{1}{q_1}=\frac{1}{p_2}+\frac{1}{q_2}=1,\nonumber
\end{align}
and in addition,
\begin{align}
&\|\frac{\partial g^k}{\partial t}\|_{L^1(0,T;W^{-m,1}(D))}\leq C,\ \ \mathrm{for\ some}\ m\geq0,\nonumber\\
&\|h^k-h^k(t,\cdot+\eta)\|_{L^{q_1}(0,T;L^{q_2}(D))}\rightarrow0,\ |\eta|\rightarrow0,\ \mathrm{uniformly\ for}\ k.
\end{align}
Then, $g^kh^k$ converges to $gh$ in the sense of distributions on $(0,T)\times D$. Here $D$ is bounded or unbounded.
\end{lemma}
\noindent{\bf The whole space case:}\\
We consider $P(\rho_\epsilon u_\epsilon),\ Pu_\epsilon$. From (\ref{pcom1}) we obtain
\begin{align}
\partial_tP(\rho_\epsilon u_\epsilon)\ \mathrm{is\ bounded\ in}\ \ L^2(0,T;H^{-1}(B_r))+L^\infty(0,T;W^{-1,1}(B_r)).\nonumber
\end{align}
So we obtain
\begin{align}
\partial_tP(\rho_\epsilon u_\epsilon)\ \mathrm{is\ bounded\ in}\ L^2(0,T;H^{-s}(B_r)),\ s>1+\frac{\mathbb{N}}{2}.\nonumber
\end{align}
It is easy to obtain that
\begin{align}
&P(\rho_\epsilon u_\epsilon)\ \mathrm{is\ bounded\ in}\ L^2(0,T;L^s(B_r)),\nonumber\\
&Pu_\epsilon\ \mathrm{is\ bounded\ in}\ L^2(0,T;L^t(R^{\mathbb{N}})),\nonumber
\end{align}
where $s>\gamma,t>2$ if $\mathbb{N}=2$ and $s=\frac{6\gamma}{\gamma+6},t=6$ if $\mathbb{N}=3$.
Using $\frac{\gamma+6}{6\gamma}+\frac{1}{6}<1$ and Lemma \ref{lemma2}, we obtain $P(\rho_\epsilon u_\epsilon)\cdot Pu_\epsilon$ convergence to $|u|^2$ in $D'(B_r)$. Specially, we have
\begin{align}
\int_0^T\int_{B_r} P(\rho_\epsilon u_\epsilon)\cdot Pu_\epsilon=\int_0^T\int_{B_r}|u|^2.\nonumber
\end{align}
Noticing $\rho_\epsilon\rightarrow1$ in $L^\infty(0,T;L^\gamma(D))$, we obtain
\begin{align}
\int_0^T\int_{B_r}|Pu_\epsilon|^2=\int_0^T\int_{B_r} P(\rho_\epsilon u_\epsilon)\cdot Pu_\epsilon=\int_0^T\int_{B_r}|u|^2,\qquad \epsilon\rightarrow0,\nonumber
\end{align}
and
\begin{align}
\int_0^T\int_{B_r}|Pu_\epsilon-u|^2=\int_0^T\int_{B_r}|Pu_\epsilon|^2+|u|^2-2Pu_\epsilon\cdot u\rightarrow0,\qquad as\  \epsilon\rightarrow0.\nonumber
\end{align}
That is
\begin{align}
Pu_\epsilon\rightarrow u\ \in\ L^2((0,T)\times B_r).\label{3.1}
\end{align}
{\bf The bounded domain case:}\\
Using the same method, we can easily obtain
\begin{align}
Pu_\epsilon\rightarrow u\ \in\ L^2((0,T)\times D).\label{3.2}
\end{align}

\section{The proof of Theorem \ref{whole}}\label{four}
In this section, we show that $Qu_\epsilon\rightarrow0$ in $L^2(0,T;L^p(R^{\mathbb{N}})),\ \mathrm{for\ some}\ p\geq2$. Let
\begin{align}
G_\epsilon=&-Q(div(\rho_\epsilon u_\epsilon\otimes u_\epsilon))+\mu\Delta Qu_\epsilon+\xi\nabla divu_\epsilon-\frac{a}{{\epsilon}^2}\nabla((\rho_\epsilon)^\gamma-\gamma(\rho_\epsilon-1)-1)\nonumber\\
&+\alpha Q\nabla[(F(d_\epsilon)+\frac{1}{2}|\nabla d_\epsilon|^2)I-(\nabla d_\epsilon\odot\nabla d_\epsilon)+\frac{1}{2\lambda}(d_\epsilon\otimes N_\epsilon-N_\epsilon\otimes d_\epsilon)].\nonumber
\end{align}
Let $m_\epsilon=\rho_\epsilon u_\epsilon,\ B=a\gamma$ and $\phi^\epsilon=\binom{\varphi_\epsilon}{Q(\rho_\epsilon u_\epsilon)}$. We define operator $A$ as
\begin{align}
A\phi=\binom{divm}{B\nabla\varphi},\qquad\phi=\binom{\varphi}{m},\ \varphi\in \mathbb{R},\ m\in \mathbb{R}^\mathbb{N}.\label{4.1}
\end{align}
Equations (\ref{qcom2}) and (\ref{qcom1}) read as follows:
\begin{align}
\phi^\epsilon_t+\frac{A}{\epsilon}\phi^\epsilon=\binom{0}{G_\epsilon},\label{qcom12}
\end{align}
with $\left.\phi^\epsilon\right|_{t=0}=\phi_0^\epsilon=\binom{\varphi_\epsilon^0}{m_\epsilon^0}$. From the initial conditions, we have $\varphi_\epsilon^0\in L^\gamma_2(R^{\mathbb{N}})$ and $m_\epsilon^0\in L^{\frac{2\gamma}{\gamma+1}}$. Sobolev's embedding theorem yields $\phi_0^\epsilon\in H^{-1}(R^{\mathbb{N}})$. By virtue of (\ref{2.8})--(\ref{2.11}), we have
\begin{align}
\left\|\varphi_\epsilon u_\epsilon\right\|_{L^2(0,T;L^{\frac{\mathbb{N}}{\mathbb{N}-1}}(R^{\mathbb{N}})+L^{\frac{2\mathbb{N}\gamma}{\mathbb{N}\gamma-2\gamma+2\mathbb{N}}}(R^{\mathbb{N}}))}\leq C.\nonumber
\end{align}
Thus, we have deduced $\varphi_\epsilon u_\epsilon\in L^2((0,T);H^{-1}(R^{\mathbb{N}}))$. Next, we give the following lemma:
\begin{lemma}\label{lemma3}
Let $g,\ h\in C_0^\infty(R^{\mathbb{N}})$, $\chi\in C_0^\infty(R^{\mathbb{N}}),\ supp\chi\in B_1,\ \int_{R^{\mathbb{N}}}\chi dx=1$, and $\chi_\delta(x)=\delta^{-\mathbb{N}}\chi(\frac{x}{\delta})$. Then, we have
\begin{align}
&\left\|g-g*\chi_\delta\right\|_{L^p(R^{\mathbb{N}})}\leq C_p\delta^{1-\beta}\left\|\nabla g\right\|_{L^2(R^{\mathbb{N}})},\quad \beta=\mathbb{N}\left(\frac{1}{2}-\frac{1}{p}\right),\nonumber\\
&\left\|h*\chi_\delta\right\|_{L^{t_1}(R^{\mathbb{N}})}\leq C\delta^l\left\|h\right\|_{W^{-m,t_2}(R^{\mathbb{N}})},\quad l=-m-\mathbb{N}\left(1+\frac{1}{t_1}-\frac{1}{t_2}\right).
\end{align}
\end{lemma}
\begin{proof}
By embedding theorem, we have
\begin{align}
\left\|g-g*\chi_\delta\right\|_{L^{\frac{2\mathbb{N}}{\mathbb{N}-2}}(R^{\mathbb{N}})}\leq C\left\|\nabla g\right\|_{L^2(R^{\mathbb{N}})}.\nonumber
\end{align}
And it is easy to deduce that
\begin{align}
\left\|\int_{R^{\mathbb{N}}}(g(x)-g(x-y))\chi_\delta(y)dy\right\|_{L^2(R^{\mathbb{N}})}&=\left\|\int_{R^{\mathbb{N}}}\frac{g(x)-g(x-\delta z)}{\delta z}\delta z\chi(z)dz\right\|_{L^2(R^{\mathbb{N}})}\nonumber\\
&\leq C\delta\left\|\nabla g\right\|_{L^2(R^{\mathbb{N}})}.\nonumber
\end{align}
Using the interpolation theorem, we obtain
\begin{align}
\left\|g-g*\chi_\delta\right\|_{L^s(R^{\mathbb{N}})}\leq C\delta^{1-\beta}\left\|\nabla g\right\|_{L^2(R^{\mathbb{N}})}.\nonumber
\end{align}
where $2\leq s<\infty$ if $\mathbb{N}=2$ and $2\leq s\leq6$ if $\mathbb{N}=3$. By virtue of the following equation
\begin{align}
\left\|h*\chi_\delta\right\|_{L^{t_1}(R^{\mathbb{N}})}=\left\|\mathcal{F}^{-1}\mathcal{F}h*\chi_\delta\right\|_{L^{t_1}(R^{\mathbb{N}})}=
\left\|\mathcal{F}^{-1}\hat{h}\hat{\chi_\delta}\right\|_{L^{t_1}(R^{\mathbb{N}})},\nonumber
\end{align}
we have
\begin{align}
\left\|h*\chi_\delta\right\|_{L^{t_1}(R^{\mathbb{N}})}\leq\left\|(-\Delta)^{-\frac{m}{2}}h\right\|_{L^{t_2}(R^{\mathbb{N}})}
\left\|(-\Delta)^{\frac{m}{2}}\chi_\delta\right\|_{L^{t_3}(R^{\mathbb{N}})},\nonumber
\end{align}
where $1+\frac{1}{t_1}=\frac{1}{t_2}+\frac{1}{t_3}$.\\Using potential theory, we obtain
\begin{align}
\left\|(-\Delta)^{\frac{m}{2}}\chi_\delta\right\|_{L^{t_3}(R^{\mathbb{N}})}\leq C\delta^{-m-\mathbb{N}(1-\frac{1}{t_3})},\nonumber\\
\left\|(-\Delta)^{-\frac{m}{2}}h\right\|_{L^{t_2}(R^{\mathbb{N}})}\sim\left\|h\right\|_{\dot{W}^{-m,t_2}(R^{\mathbb{N}})},\nonumber
\end{align}
which end the proof.
\end{proof}
Let us define $\mathcal{L}(t),\ t\in \mathbb{R}$ by $e^{t A}$, where $A$ is defined in (\ref{4.1}). Then $\mathcal{L}(t)\phi_0$ solve
\begin{align}
\phi_t+A\phi=0,\
\left.\phi\right|_{t=0}=\phi_0.\nonumber
\end{align}
By Strichartz's estimate, we have
\begin{align}
\left\|\mathcal{L}(t)\phi_0\right\|_{L^{p}(R;\dot{H}_q^{s-\sigma}(R^{\mathbb{N}})}\leq\left\|\phi_0\right\|_{\dot{H}^s(R^{\mathbb{N}})},\label{4.4}
\end{align}
where $s\in \mathbb{R},\ 2\leq q<\infty,\ 2\leq p<\infty$ and $\frac{2\sigma}{\mathbb{N}+1}=\frac{2}{(\mathbb{N}-1)p}=\frac{1}{2}-\frac{1}{q}$.
The following estimates come from \cite{DG}(lemma 3.2).
\begin{align}
\left\|\mathcal{L}(\frac{t}{\epsilon})\phi_0\right\|_{L^{p}(R;W^{s-\sigma,q}(R^{\mathbb{N}}))}&\leq\epsilon^{\frac{1}{p}}\left\|\phi_0\right\|_{H^s(R^{\mathbb{N}})},\label{4.5}\\
\left\|\int_0^t\mathcal{L}(\frac{t-s}{\epsilon})\phi(s)ds\right\|_{L^{p}((0,T);W^{s-\sigma,q}(R^{\mathbb{N}}))}&\leq C(1+T)\epsilon^{\frac{1}{p}}\left\|\phi\right\|_{L^p(0,T;H^s(R^{\mathbb{N}}))}.\label{4.6}
\end{align}
Using semi-group theory, one can solve (\ref{qcom12}) by $\mathcal{L}(t)$,
\begin{align}
\phi_\epsilon(t)=\mathcal{L}(\frac{t}{\epsilon})\phi_0^\epsilon+\int_0^t\mathcal{L}(\frac{t-s}{\epsilon})G_\epsilon(s)ds.\label{4.7}
\end{align}
Naturally, One can obtain
\begin{align}
\left|Qu_\epsilon\right|\leq\left|Qu_\epsilon-Qu_\epsilon*\chi_\delta\right|+\epsilon\left|Q\varphi_\epsilon u_\epsilon*\chi_\delta\right|+\left|Qm_\epsilon*\chi_\delta\right|.\label{4.8}
\end{align}
By virtue of Lemma (\ref{lemma3}), we have
\begin{align}
&\left\|Qu_\epsilon-Qu_\epsilon*\chi_\delta\right\|_{L^2(0,T;L^p(R^{\mathbb{N}}))}\leq C\delta^{1-\mathbb{N}(\frac{1}{2}-\frac{1}{p})}\left\|\nabla u_\epsilon\right\|_{L^2((0,T)\times R^{\mathbb{N}})},\label{4.9}\\
&\epsilon\left\|Q\varphi_\epsilon u_\epsilon*\chi_\delta\right\|_{L^2(0,T;L^p(R^{\mathbb{N}}))}\leq\epsilon\delta^{-1-\mathbb{N}(\frac{1}{2}-\frac{1}{p})}
\left\|\varphi_\epsilon u_\epsilon\right\|_{L^2(0,T;H^{-1}(R^{\mathbb{N}}))}.\label{4.10}
\end{align}
So it only need to estimate $\left|Qm_\epsilon*\chi_\delta\right|$. Noticing
$$G_\epsilon\in L^\infty(0,T;W^{-1,1}(R^{\mathbb{N}}))\\+L^2(0,T;H^{-1}(R^{\mathbb{N}}))+L^2((0,T)\times R^{\mathbb{N}}),$$
 we have $G_\epsilon\in L^2(0,T;H^{-s}(R^{\mathbb{N}})),\ s>3$.
It is enough for us to obtain
\begin{align}
&\epsilon\left\|Qm_\epsilon*\chi_\delta\right\|_{L^2(0,T;L^p(R^{\mathbb{N}}))}\nonumber\\
&\leq C\delta^{-1-\sigma}\left\|\mathcal{L}(\frac{t}{\epsilon})\phi_0^\epsilon\right\|_{L^q(0,T;W^{-1-\sigma,p}(R^{\mathbb{N}}))}\nonumber\\
&\quad+C\delta^{-s-\sigma}\left\|\int_0^T\mathcal{L}(\frac{t-s}{\epsilon})QG_\epsilon(s)ds\right\|_{L^q(0,T;W^{-s-\sigma,p}(R^{\mathbb{N}}))}\nonumber\\
&\leq C\delta^{-1-\sigma}\epsilon^{\frac{1}{q}}\left\|\phi_0^\epsilon\right\|_{H^{-1}(R^{\mathbb{N}})}+C\delta^{-s-\sigma}\epsilon^{\frac{1}{q}}
\left\|G_\epsilon\right\|_{L^2(0,T;H^{-s}(R^{\mathbb{N}}))}\nonumber\\
&\leq C\delta^{-s-\sigma}\epsilon^{\frac{1}{q}}.\label{4.11}
\end{align}
Choosing $\epsilon=\delta^{q(\frac{3}{2}+s)-1}$ and substituting (\ref{4.9})-(\ref{4.11}) into (\ref{4.8}), we have proved
$$Qu_\epsilon\rightarrow0\ in\ L^2(0,T;L^p(R^{\mathbb{N}})),\ \epsilon\rightarrow0,\ p\geq2,$$
and then
\begin{align}
Pu_\epsilon\rightharpoonup u\ in\ L^2((0,T)\times R^{\mathbb{N}}).\label{4.12}
\end{align}
(\ref{3.1}) and (\ref{4.12}) yield
\begin{align}
Pu_\epsilon\rightarrow u\ in\ L^2((0,T)\times R^{\mathbb{N}}).\label{4.13}
\end{align}

\section{The proof of Theorem \ref{bounded}}\label{5}
In this section, $D$ denotes to a bounded domain of $\mathbb{R}^{\mathbb{N}}$. Let $\lambda_{k,0}^2(\lambda_{k,0}>0),\ \varphi_{k,0}$ be the eigenvalues and eigenvectors of the the Laplace operator $-\Delta$ with homogenenous Neumann boundary conditions:
\begin{align}
-\Delta\varphi_{k,0}=\lambda_{k,0}^2\varphi_{k,0},\quad \left.\frac{\partial\varphi_{k,0}}{\partial n}\right|_{\partial D}=0.
\end{align}
Here if $\lambda_{k_1,0}=\lambda_{k_2,0}$, we can choose $\{\varphi_{k,0}\}$ satisfying the followings
\begin{align}
\int_{\partial D}\nabla \varphi_{k_1,0}\nabla\varphi_{k_2,0}dx=0.
\end{align}
The operator $A$ is defined by (\ref{4.1}). We denote $\pm i\lambda_{k,0}$, $\varphi_{k,0}^\pm$ to be the eigenvalues and eigenvectors of $A$ respectively.
\begin{align}
\phi_{k,0}^\pm=\binom{\varphi_{k,0}}{m_{k,0}^\pm}=\binom{\varphi_{k,0}}{\pm\frac{\nabla\varphi_{k,0}}{i\lambda_{k,0}}}.
\end{align}
Let's define an operator $A_\epsilon$ as follows:
\begin{align}
A_\epsilon\binom{\varphi}{m}=\binom{divm}{\nabla\varphi}+\epsilon\binom{0}{\mu\Delta m+\xi\nabla divm}.\label{5.3}
\end{align}
Using the operator $A_\epsilon$, (\ref{qcom2}) and (\ref{qcom1}) can be rewritten as
\begin{align}
{\binom{\varphi_\epsilon}{Qm_\epsilon}}_t-\frac{A_\epsilon^*}{\epsilon}\binom{\varphi_\epsilon}{Qm_\epsilon}=\binom{0}{QM_\epsilon}.\label{5.14}
\end{align}
Here $A_\epsilon^*$ is the adjoint of $A_\epsilon$ and
\begin{align}
M_\epsilon=&-Q(div(\rho_\epsilon u_\epsilon\otimes u_\epsilon))+\epsilon(\mu+\xi)\nabla div(\varphi_\epsilon u_\epsilon)-\frac{a}{{\epsilon}^2}\nabla((\rho_\epsilon)^\gamma-\gamma(\rho_\epsilon-1)-1)\nonumber\\
&+\alpha Q\nabla[(F(d_\epsilon)+\frac{1}{2}|\nabla d_\epsilon|^2)-(\nabla d_\epsilon\odot\nabla d_\epsilon)+\frac{1}{2\lambda}(d_\epsilon\otimes N_\epsilon-N_\epsilon\otimes d_\epsilon)].\nonumber
\end{align}
The following lemma is another version corresponding to \cite{bepn}.
\begin{lemma}\label{lemma5}
There exists approximate eigenvalues $i\lambda_{k,\epsilon,n}^\pm$ and eigenvectors $\phi_{k,\epsilon,n}^\pm=\binom{\varphi_{k,\epsilon,n}^\pm}{m_{k,\epsilon,n}^\pm}$ of $A_\epsilon$, such that
\begin{align}
&A_\epsilon\phi_{k,\epsilon,n}^\pm=i\lambda_{k,\epsilon,n}^\pm\phi_{k,\epsilon,n}^\pm+R_{k,\epsilon,n}^\pm.\label{5.22}\\
&i\lambda_{k,\epsilon,n}^\pm=\pm i\lambda_{k,0}+i\lambda_{k,1}^\pm\sqrt{\epsilon}+O(\epsilon),\quad Re(i\lambda_{k,1}^\pm)\leq0\label{5.23}
\end{align}
For any $p\geq1$, we have
\begin{align}
\left\|R_{k,\epsilon,n}^\pm\right\|_{L^p(D)}\leq C_P(\sqrt{\epsilon})^{n+\frac{1}{p}}\ and\ \left\|\phi_{k,\epsilon,n}^\pm-\phi_{k,0}^\pm\right\|_{L^p(D)}\leq C_P(\sqrt{\epsilon})^{\frac{1}{p}}.\label{5.21}
\end{align}
\end{lemma}
\begin{proof}
First (\ref{5.21}) can be deduced directly from (\ref{5.22}) and (\ref{5.23}). So it only needs to prove (\ref{5.22}) and (\ref{5.23}). We build $\phi_{k,\epsilon,n}^\pm$ in terms of $\phi_{k,0}^\pm$. We make for $\phi_{k,\epsilon,n}^\pm,\ \lambda_{k,\epsilon,n}^\pm$ in the following form:
\begin{align}
&\phi_{k,\epsilon,n}^\pm=\sum_{i=0}^n(\sqrt{\epsilon}^i\phi_{k,i}^{\pm,int}(x)+\sqrt{\epsilon}^i\phi_{k,i}^{\pm,b}(x,\frac{d(x)}{\sqrt{\epsilon}})),\label{5.4}\\
&\lambda_{k,\epsilon,n}^\pm=\sum_{i=0}^n\sqrt{\epsilon}^i\lambda_{k,i}^\pm,\label{5.5}
\end{align}
satisfying
\begin{align}
&\phi_{k,i}^{\pm,int}=\binom{\varphi_{k,i}^{\pm,int}}{m_{k,i}^{\pm,int}},\ \phi_{k,i}^{\pm,b}=\binom{\varphi_{k,i}^{\pm,b}}{m_{k,i}^{\pm,b}},\nonumber\\
&m_{k,i}^{\pm,int}+m_{k,i}^{\pm,b}=0\quad on\ \partial D.\nonumber
\end{align}
where $\phi_{k,i}^{\pm,b}$ repidly decreases to $0$ in $\zeta$ variable which defined by $\zeta=\frac{L(x)}{\sqrt{\epsilon}}$. Here $L(x)$ is a regularized distance function to the boundary $D$, which satisfies
\begin{align}
L(x)>0\ in\ D,\quad \left.L(x)\right|_{\partial D}=0,\ \left.\nabla L(x)\right|_{\partial D}=n.\nonumber
\end{align}
Setting $\phi_{k,0}^{\pm,int}=\phi_{k,0}^{\pm},\ \lambda_{k,0}^\pm=\lambda_{k,0}$ and using (\ref{5.3}), (\ref{5.4}) and (\ref{5.5}), we have
\begin{align}
&A_\epsilon\phi_{k,\epsilon,n}^\pm=A\phi_{k,0}^\pm+\epsilon\binom{0}{\mu\Delta_x m_{k,0}^\pm+\xi\nabla_x div_xm_{k,0}^\pm}+\binom{div_xm_{k,0}^{\pm,b}+\partial\zeta m_{k,0}^{\pm,b}\cdot\frac{\nabla_x L(x)}{\sqrt{\epsilon}}}{\nabla_x\varphi_{k,0}^{\pm,b}+\partial_\zeta\varphi_{k,0}^{\pm,b}\frac{\nabla_x L(x)}{\sqrt{\epsilon}}}\nonumber\\
&+\epsilon\mu\binom{0}{\Delta_xm_{k,0}^{\pm,b}+\frac{2}{\sqrt{\epsilon}}\partial_\zeta(\nabla_xm_{k,0}^{\pm,b})\nabla_x L(x)+\frac{1}{\epsilon}\partial_\zeta^2m_{k,0}^{\pm,b}|\nabla_x L(x)|^2+\frac{1}{\sqrt{\epsilon}}\partial_\zeta m_{k,0}^{\pm,b}\Delta_x L(x)}\nonumber\\
&+\epsilon\xi\binom{0}{\nabla_xdiv_xm_{k,0}^{\pm,b}+\partial_\zeta(divm_{k,0}^{\pm,b})\frac{\nabla_x L(x)}{\sqrt{\epsilon}}+\partial_\zeta(\nabla m_{k,0}^{\pm,b})^T\frac{\nabla L(x)}{\sqrt{\epsilon}}}\nonumber\\
&+\epsilon\xi\binom{0}{(\partial_\zeta^2m_{k,0}^{\pm,b}\cdot\nabla L(x))\frac{\nabla L(x)}{\epsilon}+\partial_\zeta m_{k,0}^{\pm,b}\frac{\nabla(\nabla L(x))}{\sqrt{\epsilon}}}\nonumber\\
&+A\sum_{i=0}^n\left(\sqrt{\epsilon}^i\phi_{k,i}^{\pm,int}(x)+\sqrt{\epsilon}^i\phi_{k,i}^{\pm,b}(x,\frac{d(x)}{\sqrt{\epsilon}})\right)\nonumber\\
&+\epsilon\xi\binom{0}{\nabla div\sum_{i=1}^n\left(\sqrt{\epsilon}^im_{k,0}^{\pm,int}+\sqrt{\epsilon}^im_{k,0}^{\pm,b}\right)}+
\epsilon\mu\binom{0}{\Delta\sum_{i=1}^n\left(\sqrt{\epsilon}^im_{k,0}^{\pm,int}+\sqrt{\epsilon}^im_{k,0}^{\pm,b}\right)}\label{5.6}
\end{align}
and
\begin{align}
i\lambda_{k,\epsilon,n}^\pm\phi_{k,\epsilon,n}^\pm&=i\lambda_{k,0}^\pm\left(\phi_{k,0}^\pm+\phi_{k,0}^{\pm,b}+
\Sigma_{i=1}^n\left(\sqrt{\epsilon}^i\phi_{k,i}^{\pm,int}+\sqrt{\epsilon}^i\phi_{k,i}^{\pm,b}\right)\right)\nonumber\\
&+i\left(\Sigma_{i=1}^n\sqrt{\epsilon}\lambda_{k,i}^\pm\right)\left(\phi_{k,0}^\pm+\phi_{k,0}^{\pm,b}+
\Sigma_{i=1}^n\left(\sqrt{\epsilon}^i\phi_{k,i}^{\pm,int}+\sqrt{\epsilon}^i\phi_{k,i}^{\pm,b}\right)\right).\label{5.7}
\end{align}
Using (\ref{5.6}) and (\ref{5.7}), order $\sqrt{\epsilon}^{-1}$ in the equation (\ref{5.22}) yields
\begin{align}
\partial_\zeta m_{k,0}^{\pm,b}\cdot\nabla L(x)=0,\ and\ \partial_\zeta\varphi_{k,0}^{\pm,b}\nabla L(x)=0.\label{5.8}
\end{align}
So we have $ m_{k,0}^{\pm,b}\cdot\nabla L(x)=0,\ \varphi_{k,0}^{\pm,b}=0$. Order $\sqrt{\epsilon}^{0}$ gives
\begin{align}
&\partial_\zeta\varphi_{k,1}^{\pm,b}\nabla L(x)+\mu\partial_\zeta^2m_{k,0}^{\pm,b}|\nabla L(x)|^2+(\mu+\xi)(\partial_\zeta^2m_{k,0}^{\pm,b}\cdot\nabla L(x))\nabla L(x)=i\lambda_{k,0}^{\pm,b}m_{k,0}^{\pm,b},\label{5.9}\\
&\partial_\zeta m_{k,1}^{\pm,b}\cdot\nabla L(x)+div_xm_{k,0}^{\pm,b}=0.\label{5.10}
\end{align}
Taking scalar product of (\ref{5.9}) with $\nabla L(x)$, we obtain
\begin{align}
\partial_\zeta\varphi_{k,1}^{\pm,b}=0. \nonumber
\end{align}
Then, we have
\begin{align}
\varphi_{k,1}^{\pm,b}=0,\quad
\mu|\nabla L(x)|^2\partial_\zeta^2m_{k,0}^{\pm,b}=\pm\lambda_{k,0}m_{k,0}^{\pm,b}.\nonumber
\end{align}
Observing $m_{k,0}^{\pm,int}+m_{k,0}^{\pm,b}=0$, we have
\begin{align}
m_{k,0}^{\pm,b}=-\left(m_{k,0}^{\pm,int}(\mathfrak{L}|_{\partial D}(x))\right)\exp\left(-\zeta\frac{1\pm i}{|\nabla L(x)|}\sqrt{\frac{\lambda_{k,0}}{2\mu}}\right).\nonumber
\end{align}
Here $\mathfrak{L}|_{\partial D}(x)$ denotes the $x$ in $D$ correspond to a point of $\partial D$ in term of $L(x)$. By solving (\ref{5.10}) we can obtain
the expression of $m_{k,1}^{\pm,b}\cdot\nabla L(x)$. In particular denoting $\Delta_g$ the Laplace Beltrami operator on $\partial D$, we have
\begin{align}
m_{k,1}^{\pm,b}\cdot n=-(1\pm i)\Delta_g\varphi_{k,0}\sqrt{\frac{2\mu}{\lambda_{k,0}^3}}.
\end{align}
By internal terms corresponding to internal ones, we can build $\varphi_{k,1}^{\pm,int},\ m_{k,1}^{\pm,int}$ and $\lambda_{k,1}^{\pm}$ as follows
\begin{align}
divm_{k,1}^{\pm,int}=i\lambda_{k,0}^{\pm}\varphi_{k,1}^{\pm,int}+i\lambda_{k,1}^{\pm}\varphi_{k,0},\label{5.11}\\
\nabla\varphi_{k,1}^{\pm,int}=i\lambda_{k,0}^{\pm}m_{k,1}^{\pm,int}+i\lambda_{k,1}^{\pm}m_{k,0}^\pm,\label{5.12}
\end{align}
with $\left.m_{k,1}^{\pm,int}\cdot n\right|_{\partial D}=\left.m_{k,1}^{\pm,b}\cdot n\right|_{\partial D}$. Utilizing (\ref{5.11}) and (\ref{5.12}) we have
\begin{align}
-\Delta\varphi_{k,1}^{\pm,int}=\lambda_{k,0}^2\varphi_{k,1}^{\pm,int}+2\lambda_{k,0}^\pm\lambda_{k,1}^\pm\varphi_{k,0}\label{5.13}
\end{align}
and
\begin{align}
\partial_n\varphi_{k,1}^{\pm int}=-i\lambda_{k,0}^{\pm}m_{k,1}^{\pm b},\label{5.13'}
\end{align}
with $\left.\frac{\partial\varphi_{k,1}^{\pm,int}}{\partial n}\right|_{\partial D}=-i\lambda^\pm_{k,0}\left.m_{k,1}^{\pm,b}\cdot n\right|_{\partial D}$. Taking the scalar product of (\ref{5.13}) with $\varphi_{k,0}$ yields
\begin{align}
i\lambda_{k,1}^{\pm}=-\frac{1\pm i}{2}\sqrt{\frac{\mu}{2\lambda_{k,0}^3}}\int_{\partial D}|\nabla \varphi_{k,0}|^2.
\end{align}
So we have $Re(i\lambda_{k,1}^{\pm})<0$ if $D$ satisfies $H$-conditions. When $D$ does not satisfy $H$-condition, $\Delta_g\varphi_{k,0}$ may equal to $0$ which leads to, $Re(i\lambda_{k,1}^{\pm})=0$ and $\varphi_{k,1}^{\pm,int}=0,\ m_{k,1}^{\pm,int}=0,\ m_{k,1}^{\pm,b}=0,\ m_{k,0}^{\pm,b}=0$. So the boundary layer don't exist. We end the proof by observing the system of equations which come from $(\ref{5.6})=(\ref{5.7})$ is indeterminate.
\end{proof}
Let $\mathcal{I}\in Z^+$ be the set of $\{k| Re(i\lambda_{k.1}^\pm)<0\}$ and $\mathcal{J}=\{k| Re(i\lambda_{k.1}^\pm)=0\}$. Using the eigenvectors of $A$, we expand $Qu_\epsilon$
\begin{align}
Qu_\epsilon=\sum_{k\in Z^+}\left<Qu_\epsilon,\frac{\nabla \varphi_{k,0}}{\lambda_{k,0}}\right>\frac{\nabla \varphi_{k,0}}{\lambda_{k,0}},\nonumber
\end{align}
Let us split $Qu_\epsilon$ as
\begin{align}
Q_1u_\epsilon=\sum_{k\in \mathcal{I}}\left<Qu_\epsilon,\frac{\nabla \varphi_{k,0}}{\lambda_{k,0}}\right>\frac{\nabla \varphi_{k,0}}{\lambda_{k,0}},\ Q_2u_\epsilon=Qu_\epsilon-Q_1u_\epsilon.\nonumber
\end{align}
The difficulty of our problem is to compute the term $\rho_\epsilon u_\epsilon\times u_\epsilon$. In order to make it clear, we split it.
\begin{align}
div(\rho_\epsilon u_\epsilon\otimes u_\epsilon)&=div(m_\epsilon\otimes u_\epsilon)=div(Pm_\epsilon+Qm_\epsilon)\otimes(Pu_\epsilon+Qu_\epsilon)\nonumber\\
&=div(m_\epsilon\otimes Pu_\epsilon)+div(Pm_\epsilon\otimes Qu_\epsilon)+div(Qm_\epsilon\otimes Qu_\epsilon).\nonumber
\end{align}
Here we know that
\begin{align}
&div(m_\epsilon\otimes Pu_\epsilon)\rightharpoonup div(u\otimes u)\ in\ D'((0,T)\times D),\nonumber\\
&div(Pm_\epsilon\otimes Qu_\epsilon)\rightharpoonup0\  in\ D'((0,T)\times D).\nonumber
\end{align}
We will show that $Q_1u_\epsilon$ converges to 0 in $L^2((0,T)\times D)$ and $div(Q_2u_\epsilon\times Q_2u_\epsilon)$ is a gradient which disappears in the pressure term.\\
{\bf 1. The case of $k\in \mathcal{I}$}\\
Observing $u_\epsilon$ is bounded in $L^2((0,T);H_0^1(D))$, our problem could reduce to a finite number of modes. Indeed, we have
\begin{align}
\sum_{k>K}\int_0^T\left|\left<Qu_\epsilon,\frac{\nabla \varphi_{k,0}}{\lambda_{k,0}}\right>\right|^2dt\leq\frac{C}{\lambda_{K+1}^2}|\nabla u_\epsilon|_{L^2((0,T)\times D)}^2.\nonumber
\end{align}
$\lambda_K\rightarrow\infty\ as\ K\rightarrow\infty$. We only need to consider $\left<Qu_\epsilon,m_{k,0}^\pm\right>$ for fixed $k$.
Using $Qu_\epsilon=Q(\rho_\epsilon u_\epsilon)-\epsilon Q(\varphi_\epsilon u_\epsilon)$, we have
\begin{align}
\left<Qu_\epsilon,\frac{\nabla \varphi_{k,0}}{\lambda_{k,0}}\right>=\left<Q\rho_\epsilon u_\epsilon,\frac{\nabla \varphi_{k,0}}{\lambda_{k,0}}\right>-\epsilon\left<Q\varphi_\epsilon u_\epsilon,\frac{\nabla \varphi_{k,0}}{\lambda_{k,0}}\right>,\label{5.16}
\end{align}
It is easy to obtain that
\begin{align}
\epsilon\left<Q\varphi_\epsilon u_\epsilon,\frac{\nabla \varphi_{k,0}}{\lambda_{k,0}}\right>\leq\epsilon C\|\varphi_\epsilon u_\epsilon\|_{L^s(D)}\leq\epsilon C\|\varphi_\epsilon\|_{L^\kappa(D)}\|\nabla u_\epsilon\|_{L^2(D)}.\nonumber
\end{align}
where $s=\frac{6\kappa}{6+\kappa}$ if $\mathbb{N}=3$ and $s>\kappa$ if $\mathbb{N}=3$, $\kappa=min(2,\gamma)$.
So we have $\epsilon\left<Q\varphi_\epsilon u_\epsilon,\frac{\nabla \varphi_{k,0}}{\lambda_{k,0}}\right>$ converges to 0 in $L^2(0,T)$. We can use $\left<Qm_\epsilon,\frac{\nabla \varphi_{k,0}}{\lambda_{k,0}}\right>$ instead of $\left<Qu_\epsilon,\frac{\nabla \varphi_{k,0}}{\lambda_{k,0}}\right>$. Denoting $a_{k,\epsilon}^\pm=\left<Qu_\epsilon,\pm\frac{\nabla \varphi_{k,0}}{i\lambda_{k,0}}\right>$ and $ b_{k,\epsilon}^\pm=\left<\phi_\epsilon,\phi_{k,0}^\pm\right>$, we have $2a_{k,\epsilon}^\pm=b_{k,\epsilon}^\pm-b_{k,\epsilon}^\mp$.
In Lemma \ref{lemma5} with $n=2$, we have
\begin{align}
\left<\phi_\epsilon,\phi_{k,\epsilon,2}^\pm-\phi_{k,0}^\pm\right>\leq\epsilon^{\frac{1}{\theta'}}(\|m_\epsilon\|_{L^\theta(D)}+\|\varphi_\epsilon\|_{L^\theta(D)})\rightarrow0\ as\ \epsilon\rightarrow0,\ \theta=\min(2,\gamma,\frac{2\gamma}{\gamma+1}).\nonumber
\end{align}
We use $\left<\phi_\epsilon,\phi_{k,\epsilon,n}^\pm\right>$ instead of $\left<\phi_\epsilon,\phi_{k,0}^\pm\right>$. Applying $\phi_{k,\epsilon,n}^\pm$ on both sides of (\ref{5.14}), we obtain
\begin{align}
\frac{d}{dt}b_{k,\epsilon}^\pm(t)-\frac{\overline{i\lambda_{k,\epsilon,2}^\pm}}{\epsilon}b_{k,\epsilon}^\pm(t)=c_{k,\epsilon}^\pm(t),\label{5.15}
\end{align}
where $c_{k,\epsilon}^\pm(t)=\left<M_\epsilon,m_{k,\epsilon,2}^\pm\right>+\epsilon\left<\phi_\epsilon,R_{k,\epsilon,2}^\pm\right>$. one can solve (\ref{5.15}) as
\begin{align}
b_{k,\epsilon}^\pm(t)=b_{k,\epsilon}^{\pm,0}\exp(\overline{i\lambda_{k,\epsilon,2}^\pm}\ \frac{t}{\epsilon})+
\int_0^tc_{k,\epsilon}^\pm(s)\exp(\overline{i\lambda_{k,\epsilon,2}^\pm}\ \frac{t-s}{\epsilon})ds.\label{5.16}
\end{align}
We estimate the term $c_{k,\epsilon}^\pm(t)$.
\begin{align}
&c_{k,\epsilon}^\pm
=\int_D\left\{-Q(div(\rho_\epsilon u_\epsilon\otimes u_\epsilon))+\epsilon\mu\Delta(\varphi_\epsilon u_\epsilon)+\epsilon(\mu+\xi)\nabla div(\varphi_\epsilon u_\epsilon)\right\}m_{k,\epsilon,2}^\pm dx\nonumber\\
&-\int_D\frac{a}{{\epsilon}^2}\nabla((\rho_\epsilon)^\gamma-\gamma(\rho_\epsilon-1)-1)m_{k,\epsilon,2}^\pm dx\nonumber\\
&+\int_D\left\{\alpha Q\nabla[(F(d_\epsilon)+\frac{1}{2}|\nabla d_\epsilon|^2)-(\nabla d_\epsilon\odot\nabla d_\epsilon)+\frac{1}{2\lambda}(d_\epsilon\otimes N_\epsilon-N_\epsilon\otimes d_\epsilon)]\right\}m_{k,\epsilon,2}^\pm dx\nonumber\\
&=\sum_{i=1}^7I_i.\nonumber
\end{align}
Next we estimate each $I_i$.
\begin{align}
I_1&=\int_D-Q(div(\rho_\epsilon u_\epsilon\otimes u_\epsilon))m_{k,\epsilon,2}^\pm dx\nonumber\\
&\leq C\|\sqrt{\rho_\epsilon}u_\epsilon\|_{L^2(D)}\|\rho_\epsilon\|_{L^{2\gamma}}^{\frac{1}{2}}\|\nabla u_\epsilon\|_{L^2(D)},\nonumber\\
&\leq C,\nonumber
\end{align}
\begin{align}
I_2&=\int_D\epsilon\mu\Delta(\varphi_\epsilon u_\epsilon)m_{k,\epsilon,2}^\pm dx\nonumber\\
&\leq\epsilon\|\varphi_\epsilon\|_{L^\kappa(D)}\|\nabla u_\epsilon\|_{L^2(D)},\nonumber\\
&\leq \epsilon C,\nonumber
\end{align}
\begin{align}
I_3&=\int_D\epsilon(\mu+\xi)\nabla div(\varphi_\epsilon u_\epsilon)m_{k,\epsilon,2}^\pm dx\nonumber\\
&\leq\epsilon\|\varphi_\epsilon\|_{L^\kappa(D)}\|\nabla u_\epsilon\|_{L^2(D)},\nonumber\\
&\leq \epsilon C,\nonumber
\end{align}
\begin{align}
I_4&=\int_D-\frac{a}{{\epsilon}^2}\nabla((\rho_\epsilon)^\gamma-\gamma(\rho_\epsilon-1)-1)m_{k,\epsilon,2}^\pm dx\nonumber\\
&\leq C,\nonumber
\end{align}
\begin{align}
I_5=\int_D\nabla[(F(d_\epsilon)+\frac{1}{2}|\nabla d_\epsilon|^2)]m_{k,\epsilon,2}^\pm dx
\leq C,\nonumber
\end{align}
\begin{align}
I_6=\int_D\nabla(\nabla d_\epsilon\odot\nabla d_\epsilon)m_{k,\epsilon,2}^\pm dx
\leq C,\nonumber
\end{align}
\begin{align}
I_7&=\int_D\nabla\frac{1}{2\lambda}(d_\epsilon\otimes N_\epsilon-N_\epsilon\otimes d_\epsilon)m_{k,\epsilon,2}^\pm dx\nonumber\\
&\leq C\|d_\epsilon\|_{L^\infty(D)}\|N_\epsilon\|_{L^2(D)}.\nonumber\\
&\leq  C.\nonumber
\end{align}
Taking $\epsilon\rightarrow0$ in (\ref{5.16}), we have $b_{k,\epsilon}^{\pm}\rightarrow0$, so does $a_{k,\epsilon}^{\pm}$. Thus we have obtained $Q_1u_\epsilon\rightarrow0\ in\ L^2((0,T)\times D)$.\\
{\bf 2. The case of $k\in \mathcal{J}$.}\\
Let $\mathcal{L}(t)$ be the operator defined in Section \ref{four}. We know $\mathcal{L}(t):\ (H^s(D))^{N+1}\rightarrow(H^s(D))^{N+1}$. We claim $\mathcal{L}(-\frac{t}{\epsilon})\binom{\varphi_\epsilon}{Qm_\epsilon}$ is compact in $L^2((0,T);H^{-s}(D))$ for some $s\in(0,1)$. In order to prove the claim, we firstly have $\mathcal{L}(-\frac{t}{\epsilon})\binom{\varphi_\epsilon}{Qm_\epsilon}\in L^2((0,T);H^{-s}(D))$ as $\binom{\varphi_\epsilon}{Qm_\epsilon}\in L^2((0,T);H^{-s}(D))$. From equation (\ref{5.14}), we have $\partial_t\left[\mathcal{L}(-\frac{t}{\epsilon})\binom{\varphi_\epsilon}{Qm_\epsilon}\right]=\mathcal{L}(-\frac{t}{\epsilon})\binom{0}{M_\epsilon}$, which is bounded in $L^2((0,T);H^{-t}(D)),$ $t>\frac{\mathbb{N}}{2}+1$. Then we prove the claim
\begin{align}
\mathcal{L}(-\frac{t}{\epsilon})\binom{\varphi_\epsilon}{Qm_\epsilon}\rightarrow\binom{\varphi}{m}\ in\ L^2(0,T;H^{-s}(D)).\label{5.17}
\end{align}
By virtue of (\ref{2.7}), (\ref{2.10}) and (\ref{2.11}), one obtain
\begin{align}
\binom{\varphi_\epsilon}{Qm_\epsilon}\ &\in\ \left[L^2((0,T)\times D)\right]^{\mathbb{N}+1}+\epsilon^\theta\left[ L^2(0,T;H^{-s}(D))\right]^{\mathbb{N}+1},\label{5.18}
\end{align}
and then
\begin{align}
\binom{\varphi_\epsilon}{Qm_\epsilon}=
\mathcal{L}(-\frac{t}{\epsilon})\binom{\varphi}{m}+\binom{\widetilde{S}_\epsilon(t,x)}{\widetilde{R}_\epsilon(t,x)},\label{5.19}
\end{align}
where $\binom{\varphi}{m}\in \left[L^2((0,T)\times D)\right]^{\mathbb{N}+1}$ and $\binom{\widetilde{S}_\epsilon(t,x)}{\widetilde{R}_\epsilon(t,x)}\rightarrow0\ in\ \left[L^2(0,T;H^{-s}(D))]\right]^{\mathbb{N}+1}$.
We should only consider the solutions of $\mathcal{L}(-\frac{t}{\epsilon})\binom{\varphi}{m}$ in $span\{\phi_k\}_{k\in \mathcal{J}}$. Let $\varphi=\sum_{k\in \mathcal{J}}\widetilde{a}_k(t)\varphi_{k,0}$ and $m=\sum_{k\in \mathcal{J}}\widetilde{b}_k(t)\frac{\nabla\varphi_{k,0}}{\lambda_{k,0}}$. Then we obtain
\begin{align}
\mathcal{L}(-\frac{t}{\epsilon})\binom{\varphi}{m}=\binom{\cdots}{\sum_{k\in \mathcal{J}}\left[\cos(\sqrt{B}\lambda_{k,0})\frac{t}{\epsilon}\widetilde{b}_k(t)-
\sin(\sqrt{B}\lambda_{k,0})\frac{t}{\epsilon}\widetilde{a}_k(t)\right]\frac{\nabla\varphi_{k,0}}{\lambda_{k,0}}}.\nonumber
\end{align}
Let us denote $a_k^\epsilon(t)=\left[\cos(\sqrt{B}\lambda_{k,0})\frac{t}{\epsilon}\widetilde{b}_k(t)-
\sin(\sqrt{B}\lambda_{k,0})\frac{t}{\epsilon}\widetilde{a}_k(t)\right]\frac{1}{\lambda_{k,0}}$ and\\$v_\epsilon=\sum_{k\in \mathcal{J}}a_k^\epsilon(t)\nabla\varphi_{k,0}$. We have
\begin{align}
div(Qm_\epsilon\otimes u_\epsilon)=div\left[(v_\epsilon+\widetilde{R}_\epsilon(t,x))\otimes u_\epsilon\right]=div(v_\epsilon\otimes u_\epsilon)+
div(\widetilde{R}_\epsilon(t,x)\otimes u_\epsilon),
\end{align}
and
\begin{align}
div(\widetilde{R}_\epsilon(t,x)\otimes u_\epsilon)\rightharpoonup0\ in\ D'((0,T)\times D).\nonumber
\end{align}
We claim the following holds:
\begin{align}
div(v_\epsilon\otimes (Qu_\epsilon-v_\epsilon)\rightharpoonup0\ in\ D'((0,T)\times D).\label{5.20}
\end{align}
Indeed using (\ref{5.19}), we have
\begin{align}
\|Qu_\epsilon-v_\epsilon\|_{L^2(0,T;H^{-s}(D))}\leq \epsilon\|\varphi_\epsilon u_\epsilon\|_{L^2(0,T;H^{-s}(D))}+\|\widetilde{R}_\epsilon(t,x)\|_{L^2(0,T;H^{-s}(D))}.\nonumber
\end{align}
So $Qu_\epsilon-v_\epsilon\rightarrow0\ in\ L^2(0,T;H^{-s}(D))$. As $v_\epsilon\in L^2((0,T)\times D)$, we can find $v_\epsilon^\delta\in L^2(0,T;H^{s}(D))$ such that $\|v_\epsilon-v_\epsilon^\delta\|_{L^2((0,T)\times D)}\leq\delta$. Then we have
\begin{align}
div(v_\epsilon\otimes (Qu_\epsilon-v_\epsilon)&=div\left[(v_\epsilon-v_\epsilon^\delta)\otimes(u_\epsilon-v_\epsilon)+v_\epsilon^\delta\otimes(u_\epsilon-v_\epsilon)\right]\nonumber\\
&\leq C\delta+o(\epsilon).
\end{align}
Passing to the limit for  $\epsilon\rightarrow0$, and then $\delta\rightarrow0$, we have prove the claim (\ref{5.20}). Finally, we consider $div(v_\epsilon\otimes v_\epsilon)$.
\begin{align}
div(v_\epsilon\otimes v_\epsilon)=div\sum_{k,l\in \mathcal{J}}h_{k,l}^\epsilon(t)\nabla \varphi_{k,0}\otimes\nabla \varphi_{l,0}.
\end{align}
where $h_{k,l}^\epsilon(t)=a_k^\epsilon(t)a_l^\epsilon(t)$. If $\lambda_{k,0}\neq\lambda_{l,0}$, by Riemann-Lebesgue Lemma we have
\begin{align}
h_{k,l}^\epsilon(t)\rightharpoonup0\ in\ L^1(0,T).\nonumber
\end{align}
It only leaves with the sum on pairs of $(l,k)$, such that $\lambda_{k,0}=\lambda_{l,0}\ k,j\in \mathcal{J}$. Let denote $\mathfrak{D}=\{k|\lambda_{k,0}=\lambda_{l,0},k\in \mathcal{J},k\in \mathcal{J}\}$. One easily has
\begin{align}
div\sum_{k,j\in \mathfrak{D}}h_{k,l}^\epsilon(t)\nabla\varphi_{k,0}\otimes\nabla\varphi_{l,0}=div\sum_{k>j\in \mathfrak{D}}h_{k,l}^\epsilon(t)\left(\nabla\varphi_{k,0}\otimes\nabla\varphi_{l,0}+\nabla\varphi_{l,0}\otimes\nabla\varphi_{k,0}\right),
\end{align}
and
\begin{align}
div\left(\nabla\varphi_{k,0}\otimes\nabla\varphi_{l,0}+\nabla\varphi_{l,0}\otimes\nabla\varphi_{k,0}\right)
=-\lambda_{k,0}^2\nabla(\varphi_{k,0}\varphi_{l,0})+\nabla(\nabla\varphi_{k,0}\cdot\nabla\varphi_{l,0}).\nonumber
\end{align}
Thus we have proved that the limit of $div(v_\epsilon\otimes v_\epsilon)$ is a gradient, so does the limit of $div(Qm_\epsilon\otimes Qu_\epsilon)$, which end the proof of Theorem \ref{bounded}.\\

\noindent{\bf Acknowledgements:} We would like to thank Professor Chun Liu for his advice to this problem.

\end{document}